\documentclass[12pt]{article}

\usepackage{amsfonts,amsmath,amsthm,latexsym,amssymb,fullpage,amscd}
\usepackage[mathscr]{eucal}
\usepackage{tikz}

 at 16 true pt
 at 13  true pt
 at 12 true pt

\newcommand{\R}{\mathbb{R}}
\newtheorem{theorem}{Theorem}[section]
\newtheorem{lemma}[theorem]{Lemma}
\newtheorem{corollary}{Corollary}[theorem]
\newtheorem{definition}{Definition}[section]

\title {Fourier decay and $L^p$  Sobolev smoothing for weighted hypersurface measures in $\R^3$}

\author{Michael Greenblatt}
\date{\today}

\newcommand\blfootnote[1]{%
  \begingroup
  \renewcommand\thefootnote{}\footnote{#1}%
  \addtocounter{footnote}{-1}%
  \endgroup
}
\begin{document}
\maketitle
\begin{abstract} 

We consider local  hypersurface measures in $\R^3$ whose density is allowed to have a weight function constructed from real analytic functions in a broad sense.
 We prove $L^p$  Sobolev smoothing theorems for convolutions with such surface measures and  Fourier transform decay rate results for these measures, generalizing and
subsuming earlier results for smooth densities. Our theorems are sharp in an appropriate sense and 
 can be described in terms of relatively simple properties of the surfaces and weight functions.

\end{abstract}
\blfootnote{This work was supported by a grant from the Simons Foundation.}

\section{Background and theorem statements} 

\subsection {Overview}

Let $S$ be a hypersurface in $\R^3$ that is the graph of a real analytic function $f(x,y)$ on an open disk $D_0$ centered at the 
 origin. Rotating and translating coordinates as necessary, we assume that $f(x,y)$ is not identically zero and satisfies
\[f(0,0) = 0 {\hskip 1 in} \nabla f(0,0) = \langle 0,0 \rangle \tag{1.1}\]
We will weight the Euclidean surface measure on $S$ using a function $\phi(x,y)$ built out of real analytic functions in the following reasonably general
fashion. 

Let $q_1(x,y),...,q_N(x,y)$ be real analytic functions on $D_0$. Define $E = \cap_{j =1}^N\{(x,y) \in D_0: q_j(x,y) > 0\}$. We assume $(0,0)$ is in the closure of $ E$ to avoid
 trivialities. Note that by taking $N = 1$ and $q_1(x,y) = 1$, we are allowing the case where $E = D_0$.   For real analytic functions $\{h_j(x,y)\}_{j=1}^M$ 
on $D_0$ with $h_j(0,0) = 0$ but  none identically zero, and a $C^1$ function $\alpha(x,y)$ on $D_0$ supported on a disk $D \subsetneq D_0$ centered at the origin we define
\[ \phi(x,y) = \chi_E(x,y) \,\alpha(x,y) \prod_{j=1}^M|h_j(x,y)|^{\gamma_j} \tag{1.2}\]
For now the only assumption we make on the exponents $\gamma_j$ is that they are real and nonzero (so either positive or negative)
such that  $\int_{D \cap E} \prod_{j=1}^M|h_j(x,y)|^{\gamma_j}$ is finite. The latter condition ensures that $\phi(x,y)$ is always 
integrable over $D$, which will be needed to define our operators.

\noindent Let $\mu$ be the surface measure on $S$ described by integration against functions $g$ by
\[ \int_{\R^3} g \,d\mu  = \int_D g(x,y, f(x,y)) \,\phi(x,y)\,dx\,dy \tag{1.3}\] 
Sometimes it is more convenient to write this as 
\[ \int_{D \cap E} g(x,y, f(x,y)) \,\phi(x,y)\,dx\,dy \tag{1.4} \] 
The focus of this paper will be to find  decay estimates for the Fourier transforms of such measures as well as $L^p$ Sobolev smoothing results for convolutions with such measures.

We will always assume that $D$ is sufficiently small that 
one can use Theorem 3.1 of this paper to simultaneously resolve the singularities of $f$ and the
 various functions $h_j$, $q_j$ appearing here on $D$. In this way $D$ will be a finite union of wedges $D_i$ such that to each wedge $D_i$ there is a coordinate change after which
$D_i$ becomes a new wedge $D_i'$ on which each of these resolved functions is 
approximately a monomial. We will then decompose a given $D_i'$  into dyadic rectangles $R_{ik}$ for our analysis.

Key to proving our results will be estimates on the Fourier transform $\hat{\mu}(\lambda)$, given by
\[ \hat{\mu}(\lambda_1,\lambda_2,\lambda_3) =  \int_{D \cap E} e^{-i\lambda_1 x - i\lambda_2 y - i \lambda_3 f(x,y)} \phi(x,y)\,dx\,dy \tag{1.5}\]
The basic idea is that on a given rectangle $R_{ik}$, since the functions $f$, $h_j$, and $q_j$ are approximately monomials, one can use Van der Corput lemmas and related facts to prove
desirable estimates on the contribution to $\hat{\mu}(\lambda)$ coming from $R_{ik}$. Adding these estimates over all $k$ and $i$ will give estimates of the form $|\hat{\mu}(\lambda)| \leq C 
(1 + |\lambda|)^{-\eta} (\ln (2 + |\lambda|))^l$, where $l = 0$ or $1$ and $\eta \leq \frac{1}{3}$ that are sharp when $\eta < \frac{1}{3}$. These will be given in  Theorem 1.1.

Let $\mu_{ik}$ be the measure where  in $(1.4)$ one replaces $D \cap E$ by $R_{ik}$ in the original coordinates. 
Then the contribution to $\hat{\mu}(\lambda)$ coming from $R_{ik}$ will be equivalent to a certain $L^2$ to
 $L^2_{\frac{1}{3}}$ Sobolev estimate for $T_{ik} f = f \ast \mu_{ik}$. Interpolating these estimates with straightforward $L^p$ to $L^p$ estimates on $T_{ik}$ 
deriving from Young's inequality, $1 < p < 
\infty$, will then give $L^p$ Sobolev space smoothing results for $T_{ik}$. Adding these over all $k$ and $i$ will result in corresponding  Sobolev space smoothing results for $T f = f \ast \mu$. These will be given in Theorem 1.2, with an analogous sharpness statement.

Integrating over sets $D \cap E$ in $(1.4)$-$(1.5)$ allows us to consider the Fourier transforms of real  analytic "pieces" of surface measures. For example if one just wants to estimate the Fourier transform of the measure corresponding to the  portion of the surface 
 inside some real analytic three-dimensional region, one can use a partition of unity to reduce the integral in question into finitely many terms of the form $(1.5)$. 
Another example would be if one wants to estimate the Fourier transform of a surface measure where one wants several different 
weights on several different parts of the surface. If the different parts can be defined via real analytic functions, then one can
 write the integral in question as the sum of finitely many terms of the form $(1.5)$.

There has been quite a bit of work done previously estimating the decay rate of Fourier transforms of hypersurface measures in $\R^3$, often in the situation where $\phi(x,y)$
 is a smooth compactly supported function. For example, as we will describe below
the general stability results of Karpushkin [K1][K2] imply part 1 of Theorem 1.1 (and more) for this class of $\phi(x,y)$.
Other estimates and generalizations to smooth $f(x,y)$ are shown in [D] [IKeM] [IM]. The latter two papers use the early work [V], where nice 
geometric connections to the Newton polygon of $f(x,y)$ in certain "adapted" coordinate systems are proven.
Further cases of part 1 of Theorem 1.1  follow from [PrY] and [G1], and
other results on the decay of the Fourier transforms of two-dimensional hypersurface measures are proven in [ESa]. If one adds an appropriate nondegeneracy
condition on the phase, many of the two-dimensional results extend to higher dimensions; we refer to [CKaNo] [DeNiSa] [G5]  [V] for examples.

As for $L^p$ Sobolev improvement theorems for convolutions with two dimensional hypersurface measures, for the case of smooth $\phi(x,y)$ the paper [G7]
shows that a result stronger than Theorem 1.2 holds. Namely, when $\phi$ is smooth  one can replace the vertex $(\frac{1}{2}, \frac{1}{3})$ by  $(\frac{1}{2}, \frac{1}{2})$ and then
$L^p(\R^3)$ to $L^p_s(\R^3)$ boundedness holds in the interior of the resulting trapezoid.

It should be mentioned that this paper builds on the author's unpublished preprint [G6] and some parts of this paper are taken from this unpublished work.

\subsection{Motivation for the optimal decay rate for $|\hat{\mu}(\lambda)|$}

Suppose that $\alpha(x,y)$ is nonnegative with $\alpha(0,0) > 0$.
Let $v$ be any unit vector in $\R^3$ and examine $\hat{\mu}(tv)$ as a function of the real variable $t$. There is a general heuristic that the supremal $\epsilon$ for which one has 
an estimate of the form $|\hat{\mu}(tv)| \leq C_v |t|^{-\epsilon}$ for $|t| > 1$ should be the same as the supremal $\epsilon$ for which one has an estimate of the form
$\sup_a \mu(\{x:  x \cdot v \in [a - u, a + u]\}) \leq D_v |u|^{\epsilon}$ for $|u| < 1$ (if this supremum is less than 1), and if all goes well one can even remove the dependence on the direction $v$ and say that the  supremal
 $\epsilon$ for which one has  an estimate of the form $|\hat{\mu}(\lambda)| \leq C |\lambda|^{-\epsilon}$ for $|\lambda| > 1$ should be the same as the supremal $\epsilon$ for which one has an estimate of the form
$\sup_{a,v} \mu(\{x:  x \cdot v \in [a - u, a + u]\}) \leq D|u|^{\epsilon}$ for $|u| < 1$. Informally, the magnitude of the oscillatory integral should be comparable to  the maximal "mass" of one wavelength of the measure. We refer to the papers [BaGuZhZo] [BNW] for examples of this philosophy.

If we add the assumption that  $\phi(x,y)$ is smooth (i.e. remove $E$ and the $h_j(x,y)$),  it is a remarkable fact that the supremal $\epsilon$ for which one has an estimate
 $|\hat{\mu}(\lambda)| \leq C |\lambda|^{-\epsilon}$ is not only given by the above 
supremum of the measures of the above slabs, but also this supremum is achieved for the specific slab $\{x:  x_3  \in [-u, u]\}$ in our coordinates where $(1.1)$ holds. This is true for the
following reason. Karpushkin's stability results [K1] and [K2] can be readily used to show that the supremal $\epsilon$ for which $|\hat{\mu}(\lambda)| \leq C |\lambda|^{-\epsilon}$ 
holds is the same as the supremal $\epsilon$  where such an estimate holds in the vertical direction. Similarly, his stability results show that the supremal $\epsilon$ for which one has an estimate of the form
$\sup_{a,v} \mu(\{x:  x \cdot v \in [a-u , a + u]\}) \leq D|u|^{\epsilon}$ for $|u| < 1$ is the same as the supremal $\epsilon$ corresponding to the slabs $\{x:  x_3  \in [-u, u]\}$. Standard
resolution of singularities techniques (see [AGuV]) then show that these latter values of $\epsilon$ are the same. 

Thus when $\phi(x,y)$ is smooth and nonnegative with $\phi(0,0) > 0$, the  supremal $\epsilon$ for which one has an estimate $|\hat{\mu}(\lambda)| \leq C |\lambda|^{-\epsilon}$ for $|\lambda| > 1$ is exactly the supremal $\epsilon$ for which one has an estimate  $\mu(\{x:  x_3  \in [-u,u]\}) \leq D|u|^{\epsilon}$ for $|u| < 1$. This can be refined to even include optimal powers of $\ln |\lambda|$ and $\ln u$ respectively. 

Suppose now that $\phi(x,y)$ is of the more general form $(1.2)$, and analogous to the above assume $\alpha(x,y)$ is nonnegative with $\alpha(0,0) > 0$. 
We will see in Theorem 1.1 that if the supremal $\epsilon$ for which $\mu(\{x:  x_3  \in [-u,u]\}) \leq D|u|^{\epsilon}$
 for $|u| < 1$ is less than  $\frac{1}{3}$, then  the above characterization once again holds and this result is sharp. For general $\alpha(x,y)$ one gets this rate of decay for
$|\hat{\mu}(\lambda)|$ or better.
We will assume  a relatively easy to describe but not too restrictive compatibility condition on $E$ and the function $\prod_{j=1}^M|h_j(x,y)|^{\gamma_j}$ ; such an assumption is needed because as we will see in the second example of section 1.4, the characterization does
 not hold in full generality.  Using various interpolations on the estimates obtained in proving Theorem 1.1, in Theorem 1.2 we will prove an $L^p$ Sobolev smoothing theorem for 
convolutions with $\mu$. This result will also be sharp in an appropriate sense if the above $\epsilon$ is less than $\frac{1}{3}$ and $\alpha(x,y)$ is nonnegative with $\alpha(0,0) > 0$. 

\noindent Specifically, for a given $E$ and $f(x,y)$ we define $m_{D}(t)$ by

\[m_{D}(t) = \int_{\{(x,y) \in D \cap E: |f(x,y)| < t\}}\prod_{j=1}^M|h_j(x,y)|^{\gamma_j} \tag{1.6}\]
In other words, $m_{D}(t)$  is the $\mu$ measure of the slab $|z| < t$ when  $\alpha(x,y) = 1$. Using resolution of singularities (see [G4] for closely 
related statements), it can be shown that there are $(\eta,l)$ such that if $D$ is small enough, then as  $t \rightarrow 0$, $m_{D}(t)$ has asymptotics of the form
\[m_{D}(t) = A_D t^{\eta}|\ln t|^l + o(t^{\eta}|\ln t|^l) \tag{1.7}\]
Here $A_D > 0$, $\eta > 0$, and $l = 0$ or $1$. Theorem 1.1 will say that when $\eta < {1 \over 3}$, if $\alpha(x,y)$ is nonnegative
with $\alpha(0,0) > 0$ then the optimal decay index for $|\hat{\mu}(\lambda)|$ will be given by 
$\eta$ (and there is even a corresponding logarithmic factor). For general $\alpha(x,y)$, the decay index is $\eta$ or better. 
In Theorem 1.2 we will see that if $\eta < \frac{1}{3}$ then $\eta$ will also give the maximum $L^p$ Sobolev derivative
 improvement for $Tf = f \ast \mu$ for $p$ in an interval containing $2$ when $\alpha(x,y)$ is nonnegative with $\alpha(0,0) > 0$. Similar to above, when $\eta < \frac{1}{3}$, 
 for general $\alpha(x,y)$ one gets at least this amount of Sobolev smoothing. 

Given that we are in coordinates such that $\nabla f(0,0) = \langle 0,0 \rangle$, the function  $|\hat{\mu}(\lambda)|$ will often have slowest decay in the $(0,0,1)$ direction. 
This holds in the case of smooth nonnegative $\phi(x,y)$ with $\phi(0,0) > 0$  described above for example.  But this does not hold for all $\phi(x,y)$, even when $\alpha(x,y)$ is 
nonnegative with $\alpha(0,0) > 0$.
For example, suppose $f(x,y) = x^k + y^k$, $\alpha (x,y) = 1$, $D \cap E = D$,  and there is one $h_j(x,y)$, given by
 $x^{-1 + \epsilon}$ for small $\epsilon$. Then one only gets $C|\lambda|^{-\epsilon}$ decay in the $x$ direction, while in the vertical direction one has a decay rate of 
$C|\lambda|^{-1/k - \epsilon/k}$, which is faster for small enough $\epsilon$.

 What is going on here is that in the "usual" situation, such as the case of  smooth $\phi(x,y)$, away from the vertical direction one readily gets very fast decrease
 in $|\hat{\mu}(\lambda)|$ simply by doing repeated integrations by parts
in the definition of the integral defining $\hat{\mu}(\lambda)$ since its phase has nonvanishing gradient. However, a sharp factor like $x^{-1 + \epsilon}$ can ruin this. So in order to 
get the desired $\eta$ as our exponent we must add a condition to ensure this sort of issue does not arise. It turns out there is a simple condition that does this for us, given by the following.

\begin{definition} We say that $E$  and the function  $\prod_{j=1}^M|h_j(x,y)|^{\gamma_j}$ are compatible with $f(x,y)$ if there is a $p > \frac{1}{1 - \min(\eta, 1/3)}$ such that 
there is a disk $D$ centered at the origin such that we have
\[\int_{D \cap E}  (\prod_{j=1}^M|h_j(x,y)|^{\gamma_j})^p  < \infty \tag{1.8}\]
\end{definition}

Since $0 < \min(\eta, \frac{1}{3}) \leq \frac{1}{3}$, the quantity $ \frac{1}{1 - \min(\eta, 1/3)}$ is in $(1, \frac{3}{2}]$. So the compatibility condition is
saying that beyond $\prod_{j=1}^M|h_j(x,y)|^{\gamma_j}$ being integrable over $D \cap E$ (which we need for our operators to even be well-defined), there is a 
$p_0(\eta) \leq \frac{3}{2}$ such that this function must be in $L^p(D \cap E)$ for some $p > p_0(\eta)$.

\subsection{Theorem statements}

We first give the theorem providing bounds for $|\hat{\mu}(\lambda)|$. 

\begin{theorem}

 Suppose $E$ and $\prod_{j=1}^M|h_j(x,y)|^{\gamma_j}$ are compatible with $f(x,y)$. Then if  $D$ is a sufficiently small disk centered at the origin the following hold for some constant
$C$ independent of $\lambda$.

\begin{enumerate}

\item

\begin{enumerate} 

\item  If $\eta < {1 \over 3}$ we have an estimate
\[|\hat{\mu}(\lambda)| \leq C(1 + |\lambda|)^{-\eta} (\ln(2 + |\lambda|))^l \tag{1.9a}\]

\item If $\eta = {1 \over 3}$ then one has an additional logarithmic factor:

\[|\hat{\mu}(\lambda)| \leq C(1 + |\lambda|)^{-{\frac{1}{3}}} (\ln(2 + |\lambda|))^{l+1}\tag{1.9b}\]

\item If $\eta  > {1 \over 3}$ then one at least has 
\[|\hat{\mu}(\lambda)| \leq C(1 + |\lambda|)^{-{1 \over 3}}\tag{1.9c}\]

\end{enumerate}

\item

\noindent If $\alpha(x,y)$ is nonnegative with $\alpha(0,0) > 0$, then if $\eta < {1 \over 3}$ the estimate of $(1.9a)$ is sharp
in the following sense. If $l = 0$ then $(1.9a)$ does not hold with $\eta$ replaced by any $\eta ' > \eta$, while if $l = 1$ then $(1.9a)$ does not hold with $l = 0$.

\end{enumerate}

\end{theorem}

\vskip 0.1 in

\noindent The constant $C$ in Theorem 1.1 will depend on $f(x,y)$, $\alpha(x,y)$, the $|h_j(x,y)|^{\gamma_j}$, and the $q_j(x,y)$ defining $E$. 

\noindent The following is our theorem giving Sobolev space smoothing for $Tf = f \ast \mu$. 

\vskip 0.2 in

\begin{theorem}

Suppose $E$ and $\prod_{j=1}^M|h_j(x,y)|^{\gamma_j}$ are compatible with $f(x,y)$. If $D$ is sufficiently small, then the following hold.

\begin{enumerate}

\item  Let $U$ be the open triangle with vertices $({1 \over 2}, {1 \over 3})$, $(0,0)$, and $(1,0)$, and  let $V = \{(x,y) \in U: 
 y < \eta\}$. Then $Tf = f \ast \mu$ is bounded from $L^p(\R^3)$ to $L^p_{s}(\R^{3})$ if $({1 \over p}, s) \in V$. 
 
\item  Suppose $\alpha(x,y)$ is  nonnegative with $\alpha (0,0) > 0$. Then if $1 < p < \infty$ and 
 $T$ is bounded from $L^p(\R^3)$ to $L^p_{s}(\R^3)$ we must have $s \leq \eta$.

\end{enumerate}

\end{theorem}

\begin{center}
\begin{tikzpicture}
\draw  (0,0) -- (10,0) -- (7.5,5/3) -- (2.5,5/3) -- (0,0);
\node at (10,-0.3) {(1,0)};
\node at (0,-0.3) {(0,0)};
\draw [dashed] (0,5/3) -- (2.5,5/3); 
\draw [dashed] (7.5,5/3) -- (11,5/3); 
\node at (12.5,5/3) {$ s =\eta$};
\draw [thick, ->] (0,0) -- (0,10/3); 
\node at (0,11/3) {$s$};
\draw [ ->] (0,0) -- (12,0) node [right] {${1 \over p}$};
\draw [dashed] (2.5,5/3) -- (5,10/3);
\draw [dashed] (7.5,5/3) -- (5,10/3);
\node at (5.8, 3.5) {$({1 \over 2} , {1 \over 3})$};
\node at (5,5/6) {$V$};
\end{tikzpicture}
\end{center}
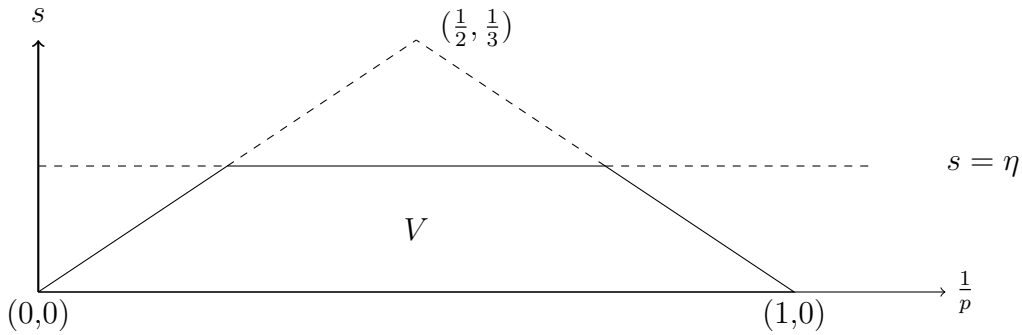
\begin{center} {Sample boundedness region when $\eta < {1 \over 3}$.}\end{center}

Note that when $\eta \geq {1 \over 3}$, $V$ is just the triangle $U$, while if $\eta < {1 \over 3}$, $V$ is a trapezoid with vertices $(0,0), ({3 \over 2}\eta, \eta), 
(1 -  {3 \over 2}\eta, \eta)$, and $(1,0)$. Thus when $\eta < {1 \over 3}$, part 2 of Theorem 1.2 shows that part 1 gives the sharp amount of $L^p(\R^3)$ to $L^p_s(\R^3)$ 
improvement up to endpoints when $\frac{1}{p} \in ({3 \over 2}\eta, 1 -  {3 \over 2}\eta)$, while if $\eta = {1 \over 3}$ the same is true for $p = 2$.

Part 2 of Theorem 1.2  is true for the following reason.  If $\alpha(x,y)$ is nonnegative with $\alpha(0,0) > 0$, then if one had
$L^p(\R^3)$ to $L^p_s(\R^3)$ Sobolev improvement for some $s > \eta$, then  by duality and interpolation one would obtain $L^2(\R^3)$ to $L^2_s(\R^3)$ Sobolev improvement and
this would contradict the last part of Theorem 1.1 since $L^2(\R^3)$ to $L^2_s(\R^3)$ Sobolev improvement is equivalent to the corresponding statement concerning the decay rate of 
$|\hat{\mu}(\lambda)|$.

\subsection{Examples}

\noindent {\bf Example.} Suppose there is just one $h_j(x,y)$, given by $h_1(x,y) = f(x,y)$, and $\phi(x,y) = |f(x,y)|^{\rho}$. We write the $m_D(t)$ of $(1.6)$ as $m_{\rho,D}(t)$ to include the dependence on $\rho$. Note that $ m_{0,D}(t)$ is just the measure of the points in $D$ for which $|f(x,y)| < t$. Observe that
\[ m_{\rho,D}(t) = \int_0^t u^{\rho} d m_{0,D}(u) \tag{1.10}\]
Thus if the main term of the asymptotics of $m_{0,D}(t)$ is of the form $A_D t^{\eta}|\ln t|^l$, then 
the main term of the asymptotics of $m_{\rho,D}(t)$ will be of the form $B_{D,\rho}t^{\eta + \rho}|\ln t|^l$. Hence if $\eta + \rho < \frac{1}{3}$, we will be in the setting of 
part 1a of Theorem 1.1  and the corresponding portion of Theorem 1.2, both of which will be sharp by the second parts of Theorem 1.1 and 1.2. On the other hand, if $\eta + \rho \geq \frac{1}{3}$ we will be in the settings of parts 1b and 1c of Theorem 1.1.

Thus we see that if $\eta + \rho < \frac{1}{3}$, the rate of decrease of $|\hat{\mu}(\lambda)|$ in this example is readily describable in terms of the measure of the points in $D$ for which $|f(x,y)| < t$.

\noindent {\bf An example related to the condition $\eta <{ 1 \over 3}$}

Consider the situation where $D \cap E = \{(x,y) \in D: x > 0,\, x^3 < y < 2x^3\}$, $f(x,y) = x^2$,
and there are two $h_j(x,y)$, given by $h_1(x,y) = x$ and $h_2(x,y) = y - x^3$.
We make no restrictions on $\gamma_1$ for now, and let $\gamma_2 = -1 + \delta$ for some small $\delta$.
 Let $\alpha(x,y) = 1$. Then we have
\[\hat{\mu}(\lambda)  = \int_{D \cap E}  x^{\gamma_1}(y- x^3)^{-1 + \delta}e^{-i\lambda_3 x^2 -i\lambda_1 x  - i\lambda_2y}\,dx\,dy  \tag{1.11}\]
Changing variables from $(x, y)$ to $(x,y + x^3)$, we get
\[\hat{\mu}(\lambda)  = \int_{\{(x,y) \in \tilde{D}:\, x > 0, \,\,0 < y < x^3\}} x^{\gamma_1}y^{-1 + \delta}e^{-i\lambda_3 x^2 - i\lambda_1 x - i\lambda_2x^3  -  i\lambda_2y}\,dx\,dy \tag{1.12}\]
Here $\tilde{D}$ is the disk $D$ in the transformed coordinates.
We look at $\hat{\mu}(\lambda)$ on rays $(\lambda_1, c_1\lambda_1, c_2\lambda_1)$ for fixed $c_1$ and $c_2$.
Then $(1.12)$ becomes
$$\hat{\mu} (\lambda_1, c_1\lambda_1, c_2\lambda_1) = \int_{\{(x,y) \in \tilde{D}:\, x > 0, \,\,0 < y < x^3\}}  x^{\gamma_1}y^{-1+\delta}
e^{-i\lambda_1 (x + c_2 x^2 + c_1 x^3)} e^{-i\lambda_1 c_1 y}\,dx\,dy \eqno (1.13)$$ When $\delta$ is very small, the $y^{ -1 + \delta}$ factor ensures that one gets very little decay in $(1.13)$ coming from the $\lambda_1 c_1y$ 
term in the exponential; the behavior is driven by the $x$ integral for fixed values of $y$. Let $r$ denote the inradius of
the disk-like $\tilde{D}$. If one chooses $c_1$ and $c_2$ so that
$x + c_2 x^2 + c_1 x^3$  has a stationary point of order 3 at some $x_0$ satisfying  ${r \over 2} < x_0 < r$, then the
best estimate one can get is $|\hat{\mu} (\lambda_1, c_1\lambda_1, c_2\lambda_1)| \leq C|\lambda_1|^{-{1 \over 3} - \delta '}$, where
$\delta '$ approaches zero as $\delta$ approaches zero. On the other hand, one may choose $\gamma_1$ so that the parameter 
$\eta$ is any given value greater than ${1 \over 3}$. 

This example illustrates the origin of the exponent ${1 \over 3}$ in the proof of part 1 of Theorem 1.1 (and therefore the corresponding aspects of Theorem 1.2), and also explains 
why one needs an additional condition if part 1 of Theorem 1.1 is to hold for exponents greater than ${1 \over 3}$. It is conceivable that the compatibility
condition $(1.8)$ used in this paper might lead to a better exponent than $\frac{1}{3}$, but this is unknown to the author. For the case of smooth $\phi(x,y)$, the papers  [K1] and [K2] show that there are no restrictions at all on $\eta$ and $(1.9a)$ always holds.

\section {Van der Corput lemmas}

\noindent We start with the well-known Van der Corput lemma (see p. 334 of [S]).

\begin{lemma}  Suppose $h(x)$ is a $C^k$ function on the interval $[a,b]$ with $|h^{(k)}(x)| > A$ on $[a,b]$ for
some $A > 0$. Let $\phi(x)$ be $C^1$ on $[a,b]$. 

\noindent If $k \geq 2$ there is a constant $c_k$ depending only on $k$ such that
$$\bigg|\int_a^b e^{ih(x)}\phi(x)\,dx\bigg| \leq c_kA^{-{1 \over k}}\bigg(|\phi(b)| + \int_a^b |\phi'(x)|\,dx\bigg) \eqno (2.1)$$
If $k =1$, the same is true if we also assume that $h(x)$ is $C^2$ and $h'(x)$ is monotone on $[a,b]$. 
\end{lemma}

\noindent  We will also make use of the following variant of Lemma 2.1 for $k = 1$, which was proven in [G3].

\begin{lemma} (Lemma 2.2 of [G3])  Suppose the hypotheses of Lemma 2.1 hold with $k = 1$, except  instead of assuming that $h'(x)$ is monotone on $[a,b]$ we assume that $|h''(x)|
< {B  \over(b-a)}A$ for some constant $B > 0$. Then we have
$$\bigg|\int_a^b e^{ih(x)}\phi(x)\,dx\bigg| \leq  A^{-1}\bigg(\int_a^b |\phi'(x)|\,dx+ (B+2) 
\sup_{[a,b]}|\phi(x)| \bigg) \eqno (2.2)$$
\end{lemma}

In section 4 we will also use the following two-dimensional mixed-derivative version of the Van der Corput Lemma from [G1].

\begin{lemma} Let $I_1$ and $I_2$ be closed intervals of lengths $l_1$ and $l_2$ respectively, and for some 
strictly monotone functions $f_1(x)$ and $f_2(x)$ on $I_1$ with $f_1(x) \leq f_2(x)$ let $R = \{(x,y) \in I_1 \times I_2: f_1(x) \leq y \leq  f_2(x )\}$ (Note $R$ might just be $I_1 \times I_2$). Suppose for some $k \geq 2$, $P(x,y)$ is a $C^k$ function on $R$ such that for each $(x,y) \in R$ one has
$$|\partial_{xy} P(x,y)| > M\,\,\,\,\,\,\,\,\,\,\,\,\,\,\,\,{\rm and }\,\,\,\,\,\,\,\,\,\,\,\,\,\,\,
\partial_y^k P(x,y) \neq 0 \eqno (2.3)$$
Further suppose that $\Psi(x,y)$ is a function on $R$ that is $C^1$ in the $y$ variable for fixed $x$,  such that 
$$ |\Psi(x,y)| < N \,\,\,\,\forall x,y\,\,\,\,\,\,\,\,\,\,\,\,\,\,\,{\rm and }\,\,\,\,\,\,\,\,\,\,\,\,\,\,\,\,\int_{\{y: (x,y) \in R\}} |\partial_y\Psi(x,y)|\,dy< N \,\,\,\,\forall x \eqno (2.4)$$
If  $R' \subset R$ such that the intersection of $R'$ with each vertical line is either empty or is a set of at most $l$ intervals, then 
the following estimate holds. 
$$\bigg|\int_{R'} e^{i P(x,y)}\Psi(x,y)\,dx\,dy\bigg| < C_{kl}  N \bigg({l_1l_2 \over M}\bigg)^{1 \over 2} \eqno (2.5)$$
\end{lemma}

\section { Resolution of singularities in two dimensions}

We will make use of the real analytic case of the resolution of singularities theorem of [G6], which is as follows. Let $S_1(x,y),...,S_k(x,y)$ be real analytic functions on a neighborhood of
 the origin with $S_j(0,0) = 0$ for each $j$, with no $S_j(x,y)$ identically zero. Let $S(x,y) = \prod_{j =1}^k S_j(x,y)$ and write  $S(x,y) = 
\sum_{\alpha,\beta} S_{\alpha\beta}x^{\alpha}y^{\beta}$.

Divide the
$xy$ plane into eight triangles by slicing the plane using  the $x$ and $y$ axes  and two lines through the origin of the form $y = \pm mx$ for some $m > 0$.
 One must ensure that these two lines are not ones  on 
which the function $\sum_{\alpha + \beta = o} S_{\alpha\beta}x^{\alpha}y^{\beta}$ vanishes other than at the
origin, where $o$ denotes the order of the zero of $S(x,y)$ at the origin.  After reflecting about the $x$ and/or $y$ axes and/or the line $y = x$ if necessary, each of the triangles becomes of the form $T_b = \{(x,y) \in \R^2: x > 0,\,0 < y < bx\}$ (modulo an inconsequential boundary set of measure zero). We will make use of the following result from [G2].

\begin{theorem}  (Theorem 2.2 of [G2]) Let  $T_b = \{(x,y) \in \R^2: x > 0,\,0 < y < bx\}$ be as above. Abusing notation slightly, use the notation $S_j(x,y)$ to denote the reflected function $S_j(\pm x,\pm y)$ or $S_j(\pm y, \pm x)$ corresponding to $T_b$.
 Then there is an $a > 0$ and a positive integer $N$ such that
if $F_a$ denotes  $\{(x,y) \in \R^2: 0 \leq x\leq a, \,0 \leq y \leq bx\}$, then one can write $F_a = \cup_{i=1}^n cl(D_i)$, such that for each $i$ there is a $k_i(x) = p_i x$ or $k_i(x) = p_i x + l_i x^{s_i} + ...$ with $k_i(x^N)$ real analytic, $l_i \neq 0$, and $s_i > 1$, such that after a coordinate change of the form $\zeta_i(x,y) = (x, \pm y + k_i(x))$, the set $D_i$ becomes a set $D_i'$ on which each function $S_j \circ \zeta_i(x,y)$ approximately becomes a monomial $d_{ij} x^{\alpha_{ij}}y^{\beta_{ij}}$, $\alpha_{ij}$ a nonnegative rational number and $\beta_{ij}$ a nonnegative integer in the following sense.

\begin{enumerate}

\item $D_i' = \{(x,y): 0 < x < a, \, g_i(x) < y < G_i(x)\}$, where $g_i(x^N)$ and $G_i(x^N)$ are
real analytic. If we expand $G_i(x) =  H_i x^{M_i} + ...$, then $M_i \geq 1$ and $H_i > 0$. The function $g_i(x)$ is either identically zero or $g_i(x) = 
h_ix^{m_i} + ...$ where $h_i > 0$ and $m_i > M_i$.

\item   The $D_i'$ can
be constructed such that for any predetermined $\delta > 0$ there is a $d_{ij} \neq 0$ such that on $D_i'$, for all $0 \leq l \leq \alpha_{ij}$ and all $0 \leq  m \leq \beta_{ij}$ one has
\[|\partial_x^l\partial_y^m(S_j \circ \zeta_i)(x,y) -  \alpha_{ij}(\alpha_{ij} - 1)...(\alpha_{ij} - l + 1)\beta_{ij}(\beta_{ij} - 1)...(\beta_{ij} - m + 1)
d_{ij}x^{\alpha_{ij} - l}y^{\beta_{ij} - m}| \]
\[\leq \delta |d_{ij}| x^{\alpha_{ij} - l}y^{\beta_{ij} - m} \tag{3.1}\]
Here the collection of $s_i$ and $(\alpha_{ij},\beta_{ij})$ is independent of $\delta$.

\end{enumerate}
 
\end{theorem}

It should be pointed out that the development of this theorem was influenced by the philosophy of [PS] where one divides a neighborhood of the origin into wedges on which $S(x,y)$ and its derivatives are well-behaved.

We also will make use of the following corollary to Theorem 3.1 that was proven in [G2].

\begin{corollary} For any predetermined $\eta > 0$ and positive integer $K$,  the $D_i'$ can be constructed, depending on $K$ and $\eta$, so that $(3.1)$  holds for all $l, m < K$, where the set of all  $s_i$ and $(\alpha_{ij},\beta_{ij})$ is independent of $\delta$ and $K$.
\end{corollary}

\section{A key estimate}

We apply Theorem 3.1 and Corollary 3.1.1 to the case where the collection of $S_j(x,y)$ are $f(x,y)$,  the functions $q_j(x,y)$ used to define $D \cap E$,  and the
 functions $h_j(x,y)$ used in the definition of the weight function $\phi(x,y)$. 

\noindent Write the integral $\hat{\mu}(\lambda_1,\lambda_2,\lambda_3)$ in $(1.5)$ as a finite sum $\sum_i U_i(\lambda_1,\lambda_2,\lambda_3)$, where 
\[U_i(\lambda_1,\lambda_2,\lambda_3) = \int_{D_i} e^{-i\lambda_3f( x, y) -  i\lambda_1 x - i\lambda_2 y}\,\phi(x,y)\,dx\,dy \tag{4.1}\]
Here $D_i$ is the domain coming from Theorem 3.1 in the current setting. Technically, $f(x,y)$ should be $f(\pm x, \pm y)$ or $f(\pm y, \pm x)$ and so on, but this does not affect the statements or proofs here. We have that either $D_i \subset E$ or $D_i \cap E = \emptyset$; since every $q_j(x,y)$ is comparable to a monomial on $D_i$,
it is either positive everywhere on $D_i$ or negative everywhere on $D_i$. If every $q_j(x,y)$ is positive on $D_i$,  then $D_i \subset E$, and otherwise  $D_i \cap E = \emptyset$.
We ignore the latter $D_i$ since $(4.1)$ is zero for them and focus our attention on bounding $(4.1)$ for the $i$ for which each $q_j(x,y)$ is positive on $D_i$ and is thus comparable to a monomial there.

We let $\zeta_i$ be the coordinate change of Theorem 3.1 and perform this coordinate change on $(4.1)$, obtaining
\[U_i(\lambda_1,\lambda_2,\lambda_3) = \int_{D_i'} e^{-i\lambda_3f_i( x, y) - i\lambda_1 x \pm i\lambda_2 y - i\lambda_2 k_i(x)} \,\phi_i(x,y)\,dx\,dy \tag{4.2}\]
Here $D_i'$ is $D_i$ in the new coordinates, 
$f_i$ denotes $f \circ \zeta_i$, and $\phi_i$ denotes $\phi \circ \zeta_i$. The function $k_i(x)$ is the $y$ translation caused by the variable
change when $x$ is fixed. To simplify notation, without loss of generality we will always assume the $\pm i\lambda_2 y$ term is of the 
form $- i\lambda_2 y$.
Let $h_{ij}(x,y)$ denote the function $h_j \circ \zeta_i(x,y)$; that is, the function $h_j(x,y)$ in the new coordinates. Theorem 3.1 says that one can write each $h_{ij}(x, y)$ in the form
 $d_{ij}x^{\alpha_{ij}}y^{\beta_{ij}}$ plus a smaller error term for some $d_{ij} \neq 0$, with similar expressions for its various partial derivatives.
Since $\phi(x,y) = \alpha(x,y) \prod_{j=1}^M|h_j(x,y)|^{\gamma_j} $ on the $D_i$ under consideration we therefore may write
\[\phi_i(x,y) =  \alpha(x, y+ k_i(x))\prod_{j=1}^M|d_{ij}x^{\alpha_{ij}}y^{\beta_{ij}} + error\,\,term|^{\gamma_j} \tag{4.3}\]
Thus we have
\[|\phi_i(x,y)| \leq C|x^{\alpha_{i1}}y^{\beta_{i1}} |^{\gamma_1}...|x^{\alpha_{iM}}y^{\beta_{iM}} |^{\gamma_M} \tag{4.4}\]
Observe that $\alpha(x, y+ k_i(x))$ is $C^1$ since $\alpha$ is and $k_i(x)$ vanishes to order at least $1$ at $x = 0$ by  Theorem 3.1. Then using Corollary 3.1.1 on the $h_j(x,y)$ we  have 
\[|\partial_{x} \phi_i(x,y)| \leq C{1 \over x} |x^{\alpha_{i1}}y^{\beta_{i1}} |^{\gamma_1}...|x^{\alpha_{iM}}y^{\beta_{iM}} |^{\gamma_M} \tag{4.5a}\]
\[|\partial_{y} \phi_i(x,y)| \leq C{1 \over y} |x^{\alpha_{i1}}y^{\beta_{i1}} |^{\gamma_1}...|x^{\alpha_{iM}}y^{\beta_{iM}} |^{\gamma_M} \tag{4.5b}\]
We now divide the integral $(4.2)$ dyadically in both $x$ and $y$. Letting $I_{lm} =  [2^{-l-1}, 2^{-l}] \times [2^{-m-1}, 2^{-m}]$, we write $U_i(\lambda_1,\lambda_2,\lambda_3) = \sum_{l,m}
U_{ilm}(\lambda_1,\lambda_2,\lambda_3)$, where
\[U_{ilm}(\lambda_1,\lambda_2,\lambda_3) = \int_{D_i'\cap I_{lm}} e^{-i\lambda_3 f_i(x,y) - i\lambda_1 x -  i\lambda_2 y
- i\lambda_2 k_i(x)}  \,\phi_i(x,y)\,dx\,dy \tag{4.6}\]
It is convenient for our arguments to combine the linear term in $k_i(x) = p_i x + l_i x^{s_i} + o(x^{s_i})$ with
the $x$ in the $\lambda_1 x$ term in $(4.6)$. So we replace $\lambda_1$ by $ \lambda_1 + p_i \lambda_2$, which will not affect any of our theorem statements. We define
$K_i(x) = k_i(x) - p_i x$, and $(4.6)$ becomes
\[U_{ilm}(\lambda_1,\lambda_2,\lambda_3) = \int_{D_i'\cap I_{lm}} e^{-i\lambda_3 f_i(x,y) - i\lambda_1 x - i\lambda_2 y
-  i\lambda_2 K_i(x)}  \,\phi_i(x,y)\,dx\,dy \tag{4.7}\]
Note that $K_i(x)$ may be the zero function. When $K_i(x)$ is nonzero, Theorem 3.1 tells us that it is of the following form, where $l_i \neq 0$ and $s_i > 1$.
\[ K_i(x) = l_i x^{s_i} + o(x^{s_i}) \tag{4.8a}\]
In what follows, we write $f_i(x,y)$ in the coordinates of $D_i'$ provided by Theorem 3.1 as 
\[ f_i(x,y) = d_i x^{\alpha_i} y^{\beta_i} + error \,\, term \tag{4.8b}\]
Throughout the remainder of the paper the constant $C$ denotes a constant independent of $\lambda$ which depends on $f(x,y)$, $\alpha(x,y)$, the $h_j(x,y)$, and the $q_j(x,y)$ defining $E$.

We will estimate $|U_{ilm}(\lambda_1,\lambda_2,\lambda_3)|$ through the use of Van der Corput lemmas in the coordinate
systems provided by Theorem 3.1. The key estimate that allows us to prove Theorem 1.1 is the following,

\noindent {\bf Theorem  4.1.} For each $i$ one has the estimate 
$$|U_{ilm}(\lambda_1,\lambda_2,\lambda_3)| \leq C\int_{I_{lm}}\min( 1, \max (|\lambda x|^{-1}, |\lambda y|^{-1}, |\lambda x^{\alpha_i}y^{\beta_i}|^{-{1 \over 3}}))
\prod_{j = 1}^M (x^{\alpha_{ij}}y^{\beta_{ij}})^{\gamma_j} \,dx\,dy \eqno (4.9)$$

\noindent {\bf Proof.} Note that we may always assume that $|\lambda| > 1$ since the case $|\lambda| \leq 1$ follows immediately from simply taking the absolute value of the 
integrand and integrating. The argument is broken into several cases, depending on $\lambda$ and the $s_i$ and $(\alpha_i,\beta_i)$ of $(4.8a)-(4.8b)$.

\noindent {\bf Case 1.} $|\lambda_3| \leq \frac{1}{2}|\lambda|$.

In this case either $|\lambda_1| \geq \frac{1}{4}|\lambda|$, $|\lambda_2| \geq \frac{1}{4}|\lambda|$, or both. Suppose $|\lambda_1| \geq \frac{1}{4}|\lambda|$ holds. 
Letting $-P(x,y)$ be the phase function in $(4.7)$, we have 
\[ {\partial P \over \partial x}(x,y) = \lambda_1 + \lambda_2 K_i'(x) + \lambda_3 {\partial f_i \over \partial x}(x,y) \tag{4.10}\]
Since $K_i(x)$ is either identically zero or has a zero of order $s_i > 1$ at $x = 0$ and $|\lambda_2| \leq |\lambda| \leq 4|\lambda_1|$ here, if the neighborhood $D$ is small enough
one will have $|\lambda_2 K_i'(x) | \leq {1 \over 4}|\lambda_1| $. Similarly, since $\nabla f(0,0) = \langle 0,0 \rangle$ by $(1.1)$, $f_i(x,y)$ has a zero of order at least $2$ at the origin and thus 
similarly if $D$ is small enough we have $|\lambda_3 {\partial f_i \over \partial x}(x,y)| \leq {1 \over 4}|\lambda_1|$. Thus if $|\lambda_1| \geq \frac{1}{4}|\lambda|$ we have
\[ \bigg|{\partial P \over \partial x}(x,y)\bigg | \geq \frac{1}{4}|\lambda_1|\]
\[ \geq \frac{1}{16}|\lambda| \tag{4.11}\]
We may now apply Lemma 2.2 in the $x$ direction with $A = \frac{1}{16}|\lambda|$ (recalling that $|\lambda| > 1$), and then integrating the result in $y$. Using 
$(4.4)$ and $(4.5a)-(4.5b)$ the resulting expression translates into
\[ |U_{ilm}(\lambda_1,\lambda_2,\lambda_3)| \leq C\int_{I_{lm}} |\lambda x|^{-1} \prod_{j = 1}^M (x^{\alpha_{ij}}y^{\beta_{ij}})^{\gamma_j} \,dx\,dy \tag{4.12a}\]
By simply taking absolute values of the integrand and integrating the result, $(4.4)$ also gives
\[ |U_{ilm}(\lambda_1,\lambda_2,\lambda_3)| \leq C\int_{I_{lm}} \prod_{j = 1}^M (x^{\alpha_{ij}}y^{\beta_{ij}})^{\gamma_j} \,dx\,dy \tag{4.12b}\]
Combining these gives
\[ |U_{ilm}(\lambda_1,\lambda_2,\lambda_3)| \leq C\int_{I_{lm}} \min(1, |\lambda x|^{-1}) \prod_{j = 1}^M (x^{\alpha_{ij}}y^{\beta_{ij}})^{\gamma_j} \,dx\,dy \tag{4.13}\]
This implies that $(4.9)$ holds, so we are done in the situation where  $|\lambda_1| \geq \frac{1}{4}|\lambda|$. If on the other hand we have $|\lambda_2| \geq \frac{1}{4}|\lambda|$,
we argue similarly but in the $y$ direction. This time we have
\[ {\partial P \over \partial y}(x,y) = \lambda_2 + \lambda_3 {\partial f_i \over \partial y}(x,y) \tag{4.14}\]
This time the fact that $f_i(x,y)$ has a zero of order at least $2$ at the origin and the fact that $|\lambda_2| \geq \frac{1}{4}|\lambda|$ imply that if $D$ is sufficiently small we have
\[ \bigg|{\partial P \over \partial y}(x,y)\bigg | \geq \frac{1}{4}|\lambda_2|\]
\[ \geq \frac{1}{16}|\lambda| \tag{4.15}\]
Then arguing as above in the $y$ direction instead of $x$ gives 
\[ |U_{ilm}(\lambda_1,\lambda_2,\lambda_3)| \leq C\int_{I_{lm}} \min(1, |\lambda y|^{-1}) \prod_{j = 1}^M (x^{\alpha_{ij}}y^{\beta_{ij}})^{\gamma_j} \,dx\,dy \tag{4.16}\]
This once again implies $(4.9)$ and thus we are done with Case 1.

\noindent The next case is the most difficult one.

\noindent {\bf Case 2.} $|\lambda_3| > \frac{1}{2}|\lambda|$, $\alpha_i > 0$, $\beta_i = 0$,  and $K_i(x)$ is not identically zero with $s_i \neq \alpha_i$.

 We will use Lemma 2.1 for either second or third $x$-derivatives and then integrate the result in $y$. Again letting $-P(x,y)$ be the phase function in $(4.7)$, we have
\[{\partial^2 P \over \partial x^2}(x,y) = \lambda_3 {\partial^2 f_i \over \partial x^2}(x,y) + \lambda_2 K_i''(x)\]
\[{\partial^3 P \over \partial x^3}(x,y) = \lambda_3 {\partial^3 f_i \over \partial x^3}(x,y) + \lambda_2 K_i'''(x)  \tag{4.17}\]
Since $K_i$ is a real analytic function of $x^{1 \over N}$ for some large $N$, $K_i(x) =  l_i x^{s_i} + 
O(x^{s_i + e})$ for some small $e > 0$ such that we also have
$$K_i''(x) = l_i s_i(s_i - 1) x^{s_i - 2} + O(x^{s_i - 2 + e}) \eqno (4.18a)$$
$$K_i'''(x) = l_i s_i(s_i - 1)(s_i - 2)x^{s_i - 3} +  O(x^{s_i - 3+ e}) \eqno (4.18b)$$
In addition, by Corollary 3.1.1 we can write
\[{\partial^2 f_i \over \partial x^2}(x,y) = d_i\alpha_i(\alpha_i - 1) x^{\alpha_i - 2} + error\,\,term \tag{4.19a}\]
\[{\partial^3 f_i \over \partial x^3}(x,y) = d_i\alpha_i(\alpha_i - 1)(\alpha_i - 2) x^{\alpha_i - 3} + error\,\, term \tag{4.19b}\]
By Corollary 3.1.1 the error term can be assumed to be of absolute value less than $\delta |d_i| x^{\alpha_i - 2}$ and $\delta |d_i| x^{\alpha_i - 3}$ respectively, for any given $\delta > 0$,
independently of the collection of all $s_i$ and $\alpha_i$.

We are assuming that $s_i \neq \alpha_i$. So the 2 by 2 matrix ${\bf M}$ with rows $(\alpha_i(\alpha_i - 1), s_i(s_i - 1) )$ and
$(\alpha_i(\alpha_i - 1)(\alpha_i - 2), s_i(s_i - 1)(s_i - 2))$ has determinant $\alpha_i(\alpha_i - 1)s_i(s_i - 1)(s_i - \alpha_i)$, which is nonzero since $\alpha_i \geq 2$, $s_i > 1$, and 
$s_i \neq \alpha_i$. 
Thus there is a constant $\rho$ such that $\|{\bf M}{\bf v}\| \geq \rho \|{\bf v}\|$ for all vectors ${\bf v} \in \R^2$. In particular,
 if ${\bf M}_p$ denotes the $p$th row of ${\bf M}$, then given ${\bf v}$ for either $p = 1$ or $2$ we have $|{\bf M}_p \cdot 
{\bf v}| > {\rho \over 2}\|{\bf v}\|$.
So if we let ${\bf v} = (\lambda_3 d_ix^{\alpha_i - 2}, \lambda_2  l_ix^{s_i - 2})$ we see that for each $x$ we either 
have the $p = 1$ case 
$$|\lambda_3  d_i\alpha_i(\alpha_i - 1) x^{\alpha_i - 2} +  \lambda_2  l_i s_i(s_i - 1)x^{s_i - 2}| \geq  {\rho \over 2}
(|\lambda_3  d_i\alpha_i(\alpha_i - 1) x^{\alpha_i - 2}|+  |\lambda_2  l_i s_i(s_i - 1)x^{s_i - 2}|)  \eqno (4.20)$$
Or the $p = 2$ case
\[|\lambda_3  d_i\alpha_i(\alpha_i - 1)(\alpha_i - 2) x^{\alpha_i - 3} +  \lambda_2  l_i s_i(s_i - 1)(s_i - 2)x^{s_i - 3}|\]
\[\geq  {\rho \over 2}
(|\lambda_3  d_i\alpha_i(\alpha_i - 1)(\alpha_i - 2) x^{\alpha_i - 3}|+  |\lambda_2  l_i s_i(s_i - 1)(s_i - 2)x^{s_i - 3}|)  \tag{4.21}\]
For the time being, assume that $\alpha_i \neq 2$ and $s_i \neq 2$ so that the coefficients in $(4.21)$ are nonzero.
Note that by looking at one higher derivative, there exists a $c_1 > 0$ depending on $s_i$ and $\alpha_i$ such that  if one of the two inequalities
 $(4.20)-(4.21)$ holds for $x = x'$, it holds for all $x \in [(1 - c_1)x', (1 + c_1)x']$ if one replaces ${\rho \over 2}$ by ${\rho \over 4}$.
Furthermore, given $(4.18)$ and $(4.19a)-(4.19b)$, if $\delta$ is small enough, which we may assume, $c_1$ may be chosen such that 
for all $(x,y) \in D_i' \cap I_{lm}$ with  $x \in [(1 - c_1)x', (1 + c_1)x']$ at least one of the two following inequalities holds.
\[\bigg|{\partial^2 P \over \partial x^2}(x,y)\bigg| \geq  {\rho \over 8}
(|\lambda_3  d_i\alpha_i(\alpha_i - 1) x^{\alpha_i - 2}|+  |\lambda_2  l_i s_i(s_i - 1)x^{s_i - 2}|) \tag{4.22a}\]
\[\bigg|{\partial^3 P \over \partial x^3}(x,y)\bigg| \geq {\rho \over 8}(|\lambda_3  d_i\alpha_i(\alpha_i - 1)(\alpha_i - 2) x^{\alpha_i - 3}|+  |\lambda_2  l_i s_i(s_i - 1)(s_i - 2)x^{s_i - 3}|)  \tag{4.22b}\]
The error terms in $(4.18a) -(4.18b)$ can be made sufficiently small for $(4.22a)$ or $(4.22b)$ to hold simply by working on a sufficiently small neighborhood of the origin.
 That the error terms in $(4.19a) -(4.19b)$ can be made small enough such that $(4.22a)$ or $(4.22b)$ hold follows from the last part of Theorem 3.1, which says that the collection of all 
$(\alpha_i,\beta_i)$ and $s_i$ can be viewed as fixed and then we can choose $\delta$ (and with it the $D_i$) accordingly 
so that $(4.22a)$ or $(4.22b)$ holds for some appropriately small $c_1$. 

The above assumed that $s_i \neq 2 \neq \alpha_i$, and now we address the possibility that $s_i = 2$ or $\alpha_i = 2$. If $(4.20)$ holds, everything proceeds as above and 
$(4.22a)$ follows. If $(4.20)$ 
does not hold for any $\rho$ depending on $s_i$ and $\alpha_i$, then for a given $\epsilon > 0$ depending on $s_i$ and $\alpha_i$  we can assume 
 that $|\lambda_3 d_i\alpha_i(\alpha_i - 1)x^{\alpha_i - 2} |$ and $|\lambda_2 l_i s_i(s_i - 1)x^{s_i - 2}|$ are within a factor of $1 + \epsilon$ of each other.
If $\epsilon$ is small enough 
this will ensure that all error terms coming from $K_i'''(x)$ in the above argument can be absorbed by the $|\lambda_3  d_i\alpha_i(\alpha_i - 1)(\alpha_i - 2) x^{\alpha_i - 3}|$ term in 
$(4.22b)$ if $s_i = 2$, and all  error terms coming from ${\displaystyle {\partial^3 f_i \over \partial x^3}}$ in the above argument can be absorbed by the 
$|\lambda_2  l_i s_i(s_i - 1)(s_i - 2)x^{s_i - 3}|$ term in $(4.22b)$ if $\alpha_i = 2$. (They can't both be $2$ since $s_i \neq \alpha_i$). Thus $(4.22b)$ holds once again as needed.

Equations $(4.22a)$ and $(4.22b)$ can be summarized by the statement 
that for any $x'$ there is a constant $\rho'  > 0$ such that for either $p = 2$ or $p = 3$, for all $(x,y) \in D_i' \cap I_{lm}$ with $x \in [(1 - c_1)x', (1 + c_1)x']$ we have
$$\bigg|{\partial^p P \over \partial x^p}(x,y)\bigg| \geq \rho' (|\lambda_3| 2^{lp} 2^{-l\alpha_i}  + 
|\lambda_2|2^{lp}2^{-ls_i}) \eqno (4.23)$$
We only need the left factor here. Namely we will only be needing that
$$\bigg|{\partial^p P \over \partial x^p}(x,y)\bigg| \geq \rho'  |\lambda_3| 2^{lp} 2^{-l\alpha_i}$$
Using that $|\lambda_3| > \frac{1}{2}|\lambda|$, this can further be rewritten as
$$\bigg|{\partial^p P \over \partial x^p}(x,y)\bigg| > 2\rho'  |\lambda| 2^{lp} 2^{-l\alpha_i}\eqno (4.24)$$

\noindent Next, we will bound a given $|U_{x'}(\lambda_1,\lambda_2,\lambda_3)|$, where 
$$U_{x'}(\lambda_1,\lambda_2,\lambda_3) = \int_{D_i'\cap( [(1 - c_1)x', (1 + c_1)x'] \times [2^{-m - 1}, 2^{-m}])} e^{-i\lambda_3 f_i(x,y) - i\lambda_1 x - i\lambda_2 y - i\lambda_2 K_i(x)} \,\phi_i(x,y)\,dx\,dy \eqno (4.25)$$
Since $D_i' \cap I_{lm}$ is the union of boundedly many sets $D_i'\cap( [(1 - c_1)x', (1 + c_1)x'] \times [2^{-m - 1}, 2^{-m}])$,
a bound for $|U_{ilm}(\lambda_1,\lambda_2,\lambda_3)|$ is given by a constant times the bound for  $|U_{x'}(\lambda_1,\lambda_2,\lambda_3)|$.

We apply Lemma 2.1 for $p$th derivatives in the $x$ direction in $(4.25)$, using $(4.24)$, and then integrate the result in the $y$ variable, using 
$(4.4)-(4.5)$ on $\phi_i(x,y)$. The result is
\[|U_{x'}(\lambda_1,\lambda_2,\lambda_3)| \leq C 2^{-l - m} |\lambda2^{-l\alpha_i}|^{-{1 \over p}} \prod_{j=1}^M(2^{-l \alpha_{ij}-m\beta_{ij}})^{\gamma_j} \]
\[\leq C  \int_{I_{lm}}(|\lambda |x^{\alpha_i})^{-{1 \over p}}
\prod_{j = 1}^M (x^{\alpha_{ij}}y^{\beta_{ij}})^{\gamma_j} \,dx\,dy  \tag{4.26} \]
Combining $(4.26)$ with the estimate one obtains by simply taking absolute values of the original integrand and 
integrating, we therefore have
\[|U_{x'}(\lambda_1,\lambda_2,\lambda_3)| \leq C \int_{I_{lm}}\min (1, (|\lambda |x^{\alpha_i})^{-{1 \over p}})
\prod_{j = 1}^M (x^{\alpha_{ij}}y^{\beta_{ij}})^{\gamma_j} \,dx\,dy  \tag{4.27}  \]
This estimate is worse when $p = 3$, so for any $x'$ we have
\[|U_{x'}(\lambda_1,\lambda_2,\lambda_3)| \leq C \int_{I_{lm}}\min (1, (|\lambda |x^{\alpha_i})^{-{1 \over 3}})
\prod_{j = 1}^M (x^{\alpha_{ij}}y^{\beta_{ij}})^{\gamma_j} \,dx\,dy  \tag{4.28}  \]
This is uniform in $x'$, so adding over boundedly many $x'$ also gives 
\[|U_{ilm}(\lambda_1,\lambda_2,\lambda_3)| \leq C \int_{I_{lm}}\min (1, (|\lambda|x^{\alpha_i})^{-{1 \over 3}})
\prod_{j = 1}^M (x^{\alpha_{ij}}y^{\beta_{ij}})^{\gamma_j} \,dx\,dy  \tag{4.29} \]
This implies  $(4.9)$, so we are done with Case 2.

\noindent {\bf Case 3.} $|\lambda_3| > \frac{1}{2}|\lambda|$, $\alpha_i > 0$, $\beta_i = 0$,  and $K_i(x)$ is not identically zero with $s_i = \alpha_i$.

The argument of Case 2 does not apply here because the matrix ${\bf M}$ has rank 1 and thus $(4.20)$ and $(4.21)$ may both fail. But if we have
$|\lambda_3 d_i + \lambda_2 l_i| \geq {1 \over 2} (|\lambda_3 d_i|  + |\lambda_2 l_i|)$ for example, then $(4.20)$ holds once again with $\rho = 1$ and then one can argue
exactly as in Case 2. So in what follows we assume that
\[|\lambda_3 d_i + \lambda_2 l_i| <{1 \over 2} (|\lambda_3 d_i|  + |\lambda_2 l_i|) \tag{4.30}\]
Because whenever two real numbers $a$ and $b$ satisfy $|a + b| < {1 \over 2}(|a| + |b|)$ one has that ${1 \over 3} |a| < |b| < 3|a|$, $(4.30)$ implies
\[ |\lambda_3| < 3 \bigg|\frac{l_i}{d_i} \lambda_2\bigg| \tag{4.31}\]
This implies that for sufficiently small $\epsilon_0 < 1$ independent of $\lambda$ we have
\[|\lambda_2| > \epsilon_0|\lambda_3| \tag{4.32}\]
Squaring this and averaging with the fact that $|\lambda_2|^2 > \epsilon_0^2 |\lambda_2|^2$ gives
\[|\lambda_2|^2 > \frac{\epsilon_0^2}{2} (|\lambda_2|^2 + |\lambda_3|^2) \tag{4.33}\]
Adding this to $|\lambda_1|^2 \geq {\displaystyle \frac{\epsilon_0^2}{2} |\lambda_1|^2}$ results in
\[ |\lambda_1|^2 + |\lambda_2|^2  >  \frac{\epsilon_0^2}{2}|\lambda|^2 \tag{4.34}\]
We may now argue as in the Case 1 situation when we had $|\lambda_1| \geq \frac{1}{4}|\lambda|$ or $|\lambda_2| \geq \frac{1}{4}|\lambda|$, replacing the 
${\displaystyle \frac{1}{4}}$ factor by ${\displaystyle \frac{\epsilon_0}{2}}$. One gets $(4.13)$ or $(4.16)$, possibly with a different constant, implying $(4.9)$ as needed. This
completes the argument for Case 3.

\noindent{\bf Case 4.} $|\lambda_3| > \frac{1}{2}|\lambda|$, $\alpha_i > 0$, $\beta_i = 0$, and $K_i(x)$ is identically zero.

Here one just uses the argument of Case 2, setting $\lambda_2 = 0$. One obtains $(4.19a)$ and therefore $(4.22a)$. One then argues as in Case 2 starting with $(4.22a)$. 

\noindent {\bf Case 5.} $|\lambda_3| > \frac{1}{2}|\lambda|$, $\alpha_i > 0$ and $\beta_i > 0$.

Let $P(x,y)$ again denote the phase function $\lambda_3 f_i(x,y) + \lambda_1 x + \lambda_2 y + \lambda_2 K_i(x)$ in $(4.7)$. Recall we are writing $f_i(x,y)$ in the
form $d_i x^{\alpha_i}y^{\beta_i}$ plus a smaller error term for some $d_i \neq 0$ and have corresponding expressions for its various partial derivatives. As a result  
${\displaystyle {\partial^2 P \over \partial x\partial y}(x,y)}$ is of the form $\alpha_i\beta_i\lambda_3d_ix^{\alpha_i-1}y^{\beta_i-1}$ plus a smaller error term. Hence if $D$ is sufficiently small, which 
we may assume,  for each $l$ and $m$, on $D_i' \cap I_{lm}$ we have
$$\bigg|{\partial^2 P \over \partial x\partial y}(x,y) \bigg| \geq C |\lambda_3| 2^{-l(\alpha_i - 1) -m(\beta_i - 1)} \eqno (4.35)$$
We now use the mixed-derivative Van der Corput lemma, Lemma 2.3, in conjunction with $(4.4)-(4.5)$ to bound $\phi_i(x,y)$ and its $y$ derivative and 
$(4.35)$ to provide lower bounds for the mixed partial derivative. One can perform a rotation before applying the resolution of singularities algorithm of Theorem 3.1
to ensure that some $k$th $y$ derivative is nonzero as is needed in Lemma 2.3 (but actually it is not too hard to show that this will hold even without such a rotation). The result is
$$|U_{ilm}(\lambda_1,\lambda_2,\lambda_3)| \leq C  |\lambda_3|^{-{1 \over 2}}(2^{{l(\alpha_i - 1) + m(\beta_i - 1)\over 2}})(2^{-{l+ m \over 2}}) 
\prod_{j = 1}^M (2^{-l \alpha_{ij}-m\beta_{ij}})^{\gamma_j} \eqno (4.36)$$
Since $h_j(x,y) \sim  x^{\alpha_{ij}}y^{\beta_{ij}}$ on $I_{lm}$,
we may write $(4.36)$ in the more convenient form
$$|U_{ilm}(\lambda_1,\lambda_2,\lambda_3)| \leq C  |\lambda_3|^{-{1 \over 2}}\int_{I_{lm}}(x^{\alpha_i}y^{\beta_i})^{-{1 \over 2}}
\prod_{j = 1}^M (x^{\alpha_{ij}}y^{\beta_{ij}})^{\gamma_j} \,dx\,dy \eqno (4.37)$$
By simply taking absolute values of the integrand and integrating, one also has 
$$|U_{ilm}(\lambda_1,\lambda_2,\lambda_3)|
\leq C \int_{I_{lm}}\prod_{j = 1}^M (x^{\alpha_{ij}}y^{\beta_{ij}})^{\gamma_j} \,dx\,dy \eqno (4.38)$$
Combining $(4.37)$ and $(4.38)$ results in
$$|U_{ilm}(\lambda_1,\lambda_2,\lambda_3)| \leq C \int_{I_{lm}}\min( 1, |\lambda_3 (x^{\alpha_i}y^{\beta_i})|^{-{1 \over 2}})
\prod_{j = 1}^M (x^{\alpha_{ij}}y^{\beta_{ij}})^{\gamma_j} \,dx\,dy \eqno (4.39)$$
This in turn implies
\[|U_{ilm}(\lambda_1,\lambda_2,\lambda_3)| \leq C \int_{I_{lm}}\min( 1, |\lambda_3 (x^{\alpha_i}y^{\beta_i})|^{-{1 \over 3}})
\prod_{j = 1}^M (x^{\alpha_{ij}}y^{\beta_{ij}})^{\gamma_j} \,dx\,dy \tag{4.40}\]
Since $|\lambda_3| > \frac{1}{2}|\lambda|$ and $f_i(x,y) \sim x^{\alpha_i}y^{\beta_i}$, equation $(4.40)$  implies $(4.9)$ and we are done with Case 5.

\noindent {\bf Case 6.} $|\lambda_3| > \frac{1}{2}|\lambda|$, $\alpha_i = 0$ and $\beta_i > 0$

\noindent Since $f(x,y)$ is assumed to have a zero of order at least two at the origin, $\beta_i > 1$ here. Since by $(3.1)$ we have ${\displaystyle {\partial^2 f_i \over \partial y^2}(x,y)} = d_i\beta_i(\beta_i - 1)y^{\beta_i - 2}$ plus a smaller error term, on $D_i' \cap I_{lm}$ we have
\[\bigg|{\partial^2 P \over \partial y^2 }(x,y)\bigg|\geq C |\lambda_3| 2^{-m(\beta_i - 2)} \tag{4.41}\]
We now use Lemma 2.1 for second derivatives in the $y$ variable and then integrate the result in $x$. Using $(4.4)-(4.5)$ to 
bound the $\phi_i(x,y)$ and its $y$ derivative, and $(4.41)$ to provide a lower bound on the second $y$ derivative, we obtain
\[|U_{ilm}(\lambda_1,\lambda_2,\lambda_3)|  \leq C|\lambda_3|^{-{1 \over 2}}(2^{m{\beta_i \over 2} - m})(2^{-l})\prod_{j = 1}^M (2^{-l \alpha_{ij}-m\beta_{ij}})^{\gamma_j} 
\tag{4.42}\]
This may be rewritten as
\[|U_{ilm}(\lambda_1,\lambda_2,\lambda_3)| \leq C  |\lambda_3|^{-{1 \over 2}}(2^{{l(\alpha_i - 1) + m(\beta_i - 1)\over 2}})(2^{-{l+ m \over 2}}) 
\prod_{j = 1}^M (2^{-l \alpha_{ij}-m\beta_{ij}})^{\gamma_j} \tag{4.43}\]
This is precisely $(4.36)$. The exact argument of Case 5 starting with $(4.36)$ now gives $(4.9)$. This completes the Case 6 argument, and therefore the proof of Theorem 4.1.

\section{The proof of Theorem 1.1}

We focus our attention on proving part 1 of Theorem 1.1. Since the estimates easily hold for $|\lambda| < 2$ by simply taking absolute values of the integrand and integrating, we 
will always assume $|\lambda| \geq  2$ and prove the estimates with $1 + |\lambda|$ replaced by $|\lambda|$ and $\ln(2 + |\lambda|)$ by $\ln|\lambda|$. 

Since $\hat{\mu}(\lambda) 
= \sum_i U_i(\lambda_1,\lambda_2,\lambda_3)$, where $ U_i(\lambda_1,\lambda_2,\lambda_3)$ is as in $(4.2)$, and there are finitely many $i$, to prove Theorem 1.1 
it suffices to show that each $|U_i(\lambda_1,\lambda_2,\lambda_3)|$ satisfies $(1.9a)-(1.9c)$. To this end, we add equation $(4.9)$ of Theorem 4.1 over all $l$ and $m$ to obtain that 
$|U_i (\lambda_1,\lambda_2,\lambda_3)|$  is bounded by
\[ C \sum_{\{(l,m): D_i' \cap I_{lm} \neq \emptyset \}}\int_{I_{lm}}\min( 1, \max (|\lambda x|^{-1}, |\lambda y|^{-1}, |\lambda x^{\alpha_i}y^{\beta_i}|^{-{1 \over 3}}))
\prod_{j = 1}^M (x^{\alpha_{ij}}y^{\beta_{ij}})^{\gamma_j} \,dx\,dy \tag{5.1}\]
Since $D_i'$ is of the form $\{(x,y): 0 < x < a,\,\,h_i(x) < y < H_i(x)\}$ with $h_i(x)$ being identically zero or having a zero of greater order
at $x = 0$ than $H_i(x)$, the shape of $D_i'$ is such that equation $(5.1)$ implies 
\[ |U_i(\lambda_1,\lambda_2,\lambda_3) | \leq C \int_{D_i'}\min( 1, \max (|\lambda x|^{-1}, |\lambda y|^{-1}, |\lambda x^{\alpha_i}y^{\beta_i}|^{-{1 \over 3}}))
\prod_{j = 1}^M (x^{\alpha_{ij}}y^{\beta_{ij}})^{\gamma_j} \,dx\,dy \tag{5.2}\]
Since $f_i(x,y) \sim x^{\alpha_i}y^{\beta_i}$ and $|h_{ij}(x,y)| \sim x^{\alpha_{ij}}y^{\beta_{ij}}$ on 
$D_i'$, (recall these are $f(x,y)$ and $h_j(x,y)$ in the coordinates of $D_i'$), the above implies that
\[ |U_i(\lambda_1,\lambda_2,\lambda_3) | \leq C \int_{D_i'}\min( 1, \max (|\lambda x|^{-1}, |\lambda y|^{-1}, |\lambda f_i(x,y)|^{-{1 \over 3}}))
\prod_{j = 1}^M |h_{ij}(x,y)|^{\gamma_j} \,dx\,dy \tag{5.3}\]
\[\leq C \int_{D_i'}\min( 1, |\lambda x|^{-1}) \prod_{j = 1}^M|h_{ij}(x,y)|^{\gamma_j} \,dx\,dy  + C \int_{D_i'}\min( 1, |\lambda y|^{-1}) \prod_{j = 1}^M|h_{ij}(x,y)|^{\gamma_j} \,dx\,dy\]
\[ + C \int_{D_i'} \min( 1, |\lambda f_i(x,y)|^{-{1 \over 3}})\prod_{j = 1}^M |h_{ij}(x,y)|^{\gamma_j} \,dx\,dy \tag{5.4}\]
Observe that
\[\min( 1, |\lambda x|^{-1}) \leq \min( 1, |\lambda x|^{-\min(\eta,\frac{1}{3})}) \leq |\lambda x|^{-\min(\eta,\frac{1}{3})}\tag{5.5a}\] 
\[\min( 1, |\lambda y|^{-1}) \leq \min( 1, |\lambda y|^{-\min(\eta,\frac{1}{3})}) \leq |\lambda y|^{-\min(\eta,\frac{1}{3})} \tag{5.5b}\]
Thus $(5.4)$ is bounded by
\[ C |\lambda|^{-\min(\eta,\frac{1}{3})} \int_{D_i'}|x|^{-\min(\eta,\frac{1}{3})} \prod_{j = 1}^M|h_{ij}(x,y)|^{\gamma_j} \,dx\,dy\]
\[  + C |\lambda|^{-\min(\eta,\frac{1}{3})} \int_{D_i'}|y|^{-\min(\eta,\frac{1}{3})}\prod_{j = 1}^M|h_{ij}(x,y)|^{\gamma_j} \,dx\,dy\]
\[ + C \int_{D_i'} \min( 1, |\lambda f_i(x,y)|^{-{1 \over 3}})\prod_{j = 1}^M |h_{ij}(x,y)|^{\gamma_j} \,dx\,dy \tag{5.6}\]
The first two terms of $(5.6)$ are done very similarly and we focus on the first term. The idea is to apply H\"older's inequality on  $|x|^{-\min(\eta,\frac{1}{3})}$ and 
$\prod_{j = 1}^M|h_{ij}(x,y)|^{\gamma_j}$. This will lead to a finite result for the $|x|^{-\min(\eta,\frac{1}{3})}$ factor if we are raising it to a factor of
less than $\frac{1}{\min(\eta,\frac{1}{3})}$. Writing $p_0 = \frac{1}{\min(\eta,\frac{1}{3})}$, its conjugate exponent is $p_0' = \frac{p_0}{p_0 - 1} = \frac{1}{1 - \min(\eta,\frac{1}{3}) }$.
But the compatibility condition $(1.8)$ ensures that there is some $t > p_0'$ such that $\int_{D \cap E}  (\prod_{j=1}^M|h_j(x,y)|^{\gamma_j})^t  < \infty$. This means that 
there is a constant $C_0$ such that $\int_{D_i'} (\prod_{j = 1}^M|h_{ij}(x,y)|^{\gamma_j})^t < C_0$. Thus if we let $p$ be the conjugate exponent to $t$, we have
$p < p_0$ and we can apply H\"older's inequality in the desired fashion for this value of $p$. Hence we see that the first term is bounded by $C'|\lambda|^{-\min(\eta,\frac{1}{3})}$. 
This is in all cases at least as good as the estimate needed in Theorem 1.1, so we are done with the analysis of this term.

The second term of $(5.6)$ can be dealt with in the same way as the first, so we move on to the third term of $(5.6)$. Converting back into the original coordinates on $D_i$ (before the resolution of singularities), this term becomes 
\[C \int_{D_i}\min( 1, |\lambda f(x,y)|^{-{1 \over 3}}) \prod_{j = 1}^M |h_j(x,y)|^{\gamma_j} \,dx\,dy \tag{5.7}\]
This is at most
\[ C \int_{D \cap E} \min( 1, |\lambda f(x,y)|^{-{1 \over 3}}) \prod_{j = 1}^M |h_j(x,y)|^{\gamma_j} \,dx\,dy \tag{5.8}\]
Letting $d\mu_h$ be the measure $\prod_{j = 1}^M |h_j(x,y)|^{\gamma_j} \,dx\,dy$, this becomes
\[ C \int_{D \cap E} \min( 1, |\lambda f(x,y)|^{-{1 \over 3}})\,d\mu_h \tag{5.9}\]
 We rewrite the integral in $(5.9)$ as
\[\mu_h(\{(x,y) \in {D \cap E} : |f(x,y)| < |\lambda|^{-1} \}) \]
\[+ |\lambda|^{-{1 \over 3}} \int_{ \{(x,y) \in {D \cap E} : |f(x,y)| > {1 \over |\lambda|}\}} |f(x,y)|^{-{1 \over 3}}\,d\mu_h  \tag{5.10}\]
By the characterization of integrals of powers of functions in terms of their distribution functions, applied to $|f(x,y)|^{-1}$ times the characteristic function of 
$|f(x,y)| > {1 \over |\lambda|}$, the integral in $(5.10)$ is equal to
\[{1 \over 3}\int_{|\lambda|^{-1}}^{\infty}t^{-{4 \over 3}}\mu_h(\{(x,y) \in {D \cap E}: |\lambda|^{-1}
< |f(x,y)| < t\})\,dt \tag{5.11}\]
Recalling the definition $(1.7)$ of $(\eta, l)$, $(5.11)$  is bounded by
\[C \int_{|\lambda|^{-1}}^{\infty}t^{-{4 \over 3}}\min(1, t^{\eta}|\ln t|^l) \,dt \tag{5.12}\]
We can put the minimum with 1 here since $\mu_h$ is a finite measure on $D \cap E$.
Given $(5.12)$ and the fact that first term in $(5.10)$ is bounded by $C|\lambda|^{-\eta}(\ln |\lambda|)^l$ by $(1.7)$, we conclude that $(5.10)$ is bounded by 
\[C|\lambda|^{-\eta}(\ln |\lambda|)^l + C |\lambda|^{-{1 \over 3}} \int_{|\lambda|^{-1}}^{\infty}t^{-{4 \over 3}}\min(1, t^{\eta}|\ln t|^l)  \,dt \tag{5.13}\]
When $\eta \leq {1 \over 3}$, we use $\min(1, t^{\eta}|\ln t|^l) \leq t^{\eta}|\ln t|^l$ in $(5.13)$. If 
$\eta < {1 \over 3}$, the result is a bound $C|\lambda|^{-\eta}(\ln |\lambda|)^l $, giving the desired estimate $(1.9a)$. If $\eta = 
{1 \over 3}$ we obtain an additional logarithmic factor in the above, giving $(1.9b)$.

If $\eta  > {1 \over 3}$, we use $1$ in the minimum if $t > 1$, and $t^{\eta}|\ln t|^l$ in the minimum if $t < 1$. In  this case we get the bound
\[ C|\lambda|^{-\eta}  + C |\lambda|^{-{1 \over 3}} \int_{|\lambda|^{-1}}^1  t^{\eta - {4 \over 3}}|\ln t|^l + C |\lambda|^{-{1 \over 3}} \int_{1}^{\infty} t^{-{4 \over 3}}  \,dt 
\tag{5.14}\]
Since $\eta > {1 \over 3}$, this is bounded by a constant times $|\lambda|^{-{1 \over 3}}$, giving us $(1.9c)$. 

Thus in all cases, the third term of $(5.6)$ satisfies the bounds of
part 1 of Theorem 1.1. Since we saw the same is true for the first two terms, we conclude that each $|U_i(\lambda_1,\lambda_2,\lambda_3)|$  satisfies the bounds of the first part 
of Theorem 1.1. As explained at the beginning of this section, this completes the proof of  Theorem 1.1 part 1.

\noindent Moving to part 2 of Theorem 1.1, as in its statement we suppose that $\alpha(x,y)$ is nonnegative with $\alpha(0,0) > 0$.
Observe that $\hat{\mu}(0,0, N)$ is given by
\[ \hat{\mu}(0,0,N) =  \int_{D} e^{ - i N f(x,y) } \phi(x,y)\,dx\,dy \tag{5.15} \]
Since $\phi(x,y) = \chi_E(x,y) \,\alpha(x,y) \prod_{j=1}^M|h_j(x,y)|^{\gamma_j}$, if we define
 \[ d\nu = \chi_E(x,y) \alpha(x,y) \prod_{j=1}^M|h_j(x,y)|^{\gamma_j}\,dx\,dy\]
then $(5.15)$ can be rewritten as 
\[ \hat{\mu}(0,0, N) =  \int_D  e^{ - i N f(x,y)} d\nu \tag{5.16}\]

\noindent We now invoke the following result of [CaCW], as stated in [Gre].

\begin{theorem} (Proposition 2.1.3  of [Gre])

Let $(\Omega, m)$ be any finite measure space, $0 < \delta < 1$, and suppose that $f : \Omega \rightarrow \R$ is a measurable function such that for all real nonzero $N$ we have
$$\left|\int_\Omega e^{i N f(x)}\,dm \right|\leq A|N|^{-\delta}$$
where $0<\delta<1$. Then for each $c\in\mathbb{R}$, we have
$$m(\{x\in\Omega:|f(x)-c|\leq t\}) \leq C_\delta At^\delta$$
where $C_\delta$ depends only on $\delta$.

\end{theorem}

We apply Theorem 5.1 here to the measure $\nu$ and $c = 0$. Suppose that we had $|\hat{\mu}(\lambda)| \leq C|\lambda|^{-\delta}$ for some $0 < \delta < 1$. Then restricting
to the vertical direction and invoking Theorem 5.1 gives 
\[ \nu(\{(x,y)\in D: |f(x,y)| < t\}) \leq C' t^{\delta} \tag{5.17} \]
In other words,
\[ \int_{\{(x,y) \in D \cap E: |f(x,y)| < t\}} \alpha(x,y) \prod_{j=1}^M|h_j(x,y)|^{\gamma_j}\,dx\,dy \leq C' t^{\delta} \tag{5.18} \]
Since $\alpha(x,y)$ is nonnegative and bounded below by a positive constant on a neighborhood of $(0,0)$, in view of $(1.7)$ we must have that $\delta \leq \eta$. Thus part 1a of Theorem 
1.1 cannot hold with $\eta$ replaced by any larger value in this setting. When $l = 1$ it cannot hold without the logarithmic factor; if it did $(5.18)$ would hold for $\delta = \eta$, contradicting 
$(1.7)$. This completes the proof of part 2 of Theorem 1.1 and therefore the proof of the whole theorem.

\section{The proof of Theorem 1.2}

As described after the statement of the theorem, part 2 of Theorem 1.2 follows readily from part 2 of Theorem 1.1, so we devote our attention to proving part 1.

\noindent {\bf Case 1.} $\eta \geq \frac{1}{3}$. 

This case follows quickly from Theorem 1.1. Parts 1b or 1c of Theorem 1.1 imply that $Tf = f \ast \mu$ is bounded from $L^2(\R^3)$ to $L^2_s(\R^3)$ for any $s < \frac{1}{3}$. 
Since $\mu$ is a finite measure, $T$ is also bounded on $L^p$ for any $1 < p < \infty$. Interpolating this with the $L^2(\R^3)$ to $L^2_s(\R^3)$  result, we get $L^p(\R^3)$ to
$L^p_s(\R^3)$ boundedness for $(\frac{1}{p},s)$ in the open triangle with vertices $(0,0)$, $(\frac{1}{2}, s)$, and $(1,0)$ for any $0 < s < \frac{1}{3}$. Taking the union of these
as $ s \rightarrow \frac{1}{3}$, gives $L^p(\R^3)$ to
$L^p_s(\R^3)$ boundedness for $(\frac{1}{p},s)$ in the open triangle with vertices $(0,0)$, $(\frac{1}{2}, \frac{1}{3})$, and $(1,0)$, which is the domain given by Theorem 1.2.

\noindent {\bf Case 2.} $\eta < \frac{1}{3}$. 

We start with the statement $(4.9)$ of Theorem 4.1. The expression $(4.9)$ for $|U_{ilm}(\lambda)|$ is bounded by
\[ C\int_{I_{lm}}\min(1, \max (|\lambda x|^{-\frac{1}{3}}, |\lambda y|^{-\frac{1}{3}}, |\lambda x^{\alpha_i}y^{\beta_i}|^{-\frac{1}{3}}))
\prod_{j = 1}^M (x^{\alpha_{ij}}y^{\beta_{ij}})^{\gamma_j} \,dx\,dy \tag{6.1}\]
We remove the left side of the minimum here, and get that $(6.1)$ is bounded by
\[ C |\lambda|^{-\frac{1}{3}} \int_{I_{lm}} \max (|x|^{-\frac{1}{3}}, |y|^{-\frac{1}{3}}, |x^{\alpha_i}y^{\beta_i}|^{-\frac{1}{3}})
\prod_{j = 1}^M (x^{\alpha_{ij}}y^{\beta_{ij}})^{\gamma_j} \,dx\,dy \tag{6.2}\]
Since $f_i(x,y) \sim x^{\alpha_i}y^{\beta_i}$ and $|h_{ij}(x,y)| \sim x^{\alpha_{ij}}y^{\beta_{ij}}$ on $D_i'$, equation $(6.2)$ is bounded by (possibly with a different constant)
\[ C |\lambda|^{-\frac{1}{3}} \int_{I_{lm}} \max (|x|^{-\frac{1}{3}}, |y|^{-\frac{1}{3}}, |f_i(x,y)|^{-\frac{1}{3}})
\prod_{j = 1}^M |h_{ij}(x,y)|^{\gamma_j} \,dx\,dy \tag{6.3}\]
Write $\mu = \sum_{ilm} \mu_{ilm}$, where $\mu_{ilm}$ is $\mu$ times the characteristic function in $(x,y)$ of the set corresponding to $I_{lm}$ in the coordinates of $D_i'$. Thus
$\widehat{\mu_{ilm}}(\lambda)$ is given by $U_{ilm}(\lambda_1,\lambda_2,\lambda_3)$. Let $T_{ilm} f = f \ast \mu_{ilm}$. Thus $T = \sum_{ilm} T_{ilm}$.  Then by $(6.3)$, one has
\[ \|T_{ilm}\|_{L^2 \rightarrow L^2_{\frac{1}{3}}} \leq C \int_{I_{lm}} \max (|x|^{-\frac{1}{3}}, |y|^{-\frac{1}{3}}, |f_i(x,y)|^{-\frac{1}{3}})
\prod_{j = 1}^M |h_{ij}(x,y)|^{\gamma_j} \,dx\,dy \tag{6.4}\]
By Young's inequality for measures (keeping in mind our coordinate changes all have Jacobian determinant 1), for $1 < p < \infty$ we also have
\[ \|T_{ilm}\|_{L^p \rightarrow L^p} \leq C \int_{I_{lm}}\prod_{j=1}^M |h_{ij}(x,y)|^{\gamma_j} \,dx\,dy \tag{6.5}\]
We now interpolate $(6.4)$ with $(6.5)$, weighting $(6.4)$ by $3\eta '$ for some $\eta ' < \eta $ and weighting $(6.5)$ by $1 - 3\eta '$. Keeping in mind that all the factors in $(6.4)$ and $(6.5)$ do not
vary by more than a uniformly bounded constant on any $I_{lm}$, the result is
\[ \|T_{ilm}\|_{L^{p_{\eta '}} \rightarrow L^{p_{\eta '}}_{\eta '}} \leq C \int_{I_{lm}} \max (|x|^{-\eta '}, |y|^{-\eta '}, |f_i(x,y)|^{-\eta '})
\prod_{j = 1}^M |h_{ij}(x,y)|^{\gamma_j} \,dx\,dy \tag{6.6}\]
Here $p_{\eta '}$ is defined by the relation
\[ \frac{1}{p_{\eta '}} = \frac{3\eta '}{2} + \frac{1 - 3 \eta '}{p} \tag{6.7}\]
Next, we observe that the right-hand side of $(6.6)$ is bounded by
\[ C \int_{I_{lm}} |x|^{-\eta '} \prod_{j = 1}^M |h_{ij}(x,y)|^{\gamma_j} \,dx\,dy +  C \int_{I_{lm}} |y|^{-\eta '}  \prod_{j = 1}^M |h_{ij}(x,y)|^{\gamma_j} \,dx\,dy \]
\[+ C \int_{I_{lm}} |f_i(x,y)|^{-\eta '} \prod_{j = 1}^M |h_{ij}(x,y)|^{\gamma_j} \,dx\,dy \tag{6.8}\]
Letting $T_i = \sum_{l,m} T_{ilm}$, we thus have
\[ \|T_i\|_{L^{p_{\eta '}} \rightarrow L^{p_{\eta '}}_{\eta '}} \leq  C \int_{D_i'} |x|^{-\eta '} \prod_{j = 1}^M |h_{ij}(x,y)|^{\gamma_j} \,dx\,dy +  C \int_{D_i'} |y|^{-\eta '}  \prod_{j = 1}^M |h_{ij}(x,y)|^{\gamma_j} \,dx\,dy \]
\[+ C \int_{D_i'} |f_i(x,y)|^{-\eta '} \prod_{j = 1}^M |h_{ij}(x,y)|^{\gamma_j} \,dx\,dy \tag{6.9}\]
The reason we can replace the union of the dyadic rectangles $I_{lm}$ with the smaller $D_i'$ in $(6.9)$ is that $D_i'$ is of the form $\{(x,y): 0 < x < a,\,\,h_i(x) < y < H_i(x)\}$ with $h_i(x)$ being identically zero or having a zero of greater order at $x = 0$ than $H_i(x)$. This shape in conjunction with the fact that the $f_i$ and $h_{ij}$ are comparable to monomials means that doing this
replacement will at most change the constant $C$ appearing here.

Following $(5.6)$ we showed using the compatibility condition $(1.8)$ that the first two terms of $(6.9)$ with $\eta '$ replaced by $\eta$ are finite. Since $\eta ' < \eta$ the 
first two terms here must be finite too. As for the third term of $(6.9)$, returning to the original coordinates turns it into
\[ C\int_{D_i} |f(x,y)|^{-\eta '} \prod_{j = 1}^M |h_{j}(x,y)|^{\gamma_j} \,dx\,dy \tag{6.10}\]
\[ \leq C\int_{D \cap E} |f(x,y)|^{-\eta '} \prod_{j = 1}^M |h_{j}(x,y)|^{\gamma_j} \,dx\,dy \tag{6.11}\]
\[ = C \sum_{k = 0}^\infty \int_{\{(x,y) \in D \cap E: \, 2^{-k-1} \leq |f(x,y)| < 2^{-k}\}} |f(x,y)|^{-\eta '} \prod_{j = 1}^M |h_{j}(x,y)|^{\gamma_j} \,dx\,dy \tag{6.12}\]
\[ \leq C' \sum_{k = 0}^\infty 2^{k \eta '} \int_{\{(x,y) \in D \cap E: \,2^{-k-1} \leq |f(x,y)| < 2^{-k}\}} \prod_{j = 1}^M |h_{j}(x,y)|^{\gamma_j} \,dx\,dy \tag{6.13}\]
Inserting $(1.7)$ into $(6.13)$ gives that the above is bounded by
\[ C'' \sum_{k = 0}^\infty k 2^{k(\eta ' - \eta)} \tag{6.14} \]
Since $\eta' < \eta$, the sum $(6.14)$ is finite. We conclude that the final term of $(6.9)$ is finite.  Now we have seen that the three terms of 
$(6.9)$ are all finite, so we may conclude that $T_i$ is bounded from $L^{p_{\eta '}}$ to $L^{p_{\eta '}}_{\eta '}$. Adding this over all the finitely many $i$ at last gives that $T$ itself is bounded from $L^{p_{\eta'}}$ to $L^{p_{\eta '}}_{\eta '}$. 

Note that as the $p$ in $(6.5)$ converges to $1$ and $\eta'$ converges to $\eta$, by $(6.7)$ we have that $({1 \over p_{\eta'}}, \eta')$ converges to $(1 - \frac{3 \eta}{2}, \eta)$, the
upper right corner of the trapezoid $V$ of Theorem 1.2. Similarly, as $p$  converges to $\infty$ and $\eta'$ converges to $\eta$, equation $(6.7)$ gives that $({1 \over p_{\eta'}}, \eta')$ converges
to $(\frac{3 \eta}{2}, \eta)$, the upper left corner of $V$. Since one can interpolate any $L^p$ to $L^{p_{\eta '}}_{\eta '}$ boundedness with any $L^q$ to $L^q$ boundedness
for $1 < q < \infty$ given that $T$ is a convolution with a finite measure, we conclude that we have $L^p$ to $L^p_s$ boundedness for any $(\frac{1}{p}, s)$ in the interior of $V$.
This is exactly the statement of Theorem 1.2 for Case 2 and we are done.

\section{References}

\noindent [AGuV] V. Arnold, S. Gusein-Zade, A. Varchenko, {\it Singularities of differentiable maps},
Volume II, Birkh\"auser, Basel, 1988. \parskip = 4pt\baselineskip = 3pt

\noindent [BaGuZhZo] S. Basu, S. Guo, R. Zhang, P. Zorin-Kranich, {\it A stationary set method for estimating oscillatory integrals},
J. Eur. Math. Soc. (JEMS) {\bf 28} (2026), no. 5, 2067-2100.

\noindent [BNW] J. Bruna, A. Nagel, and S. Wainger, {\it Convex hypersurfaces and Fourier transforms}, Ann. of Math. (2) {\bf 127} no. 2, (1988), 333--365. 

\noindent [CaCW] A. Carbery, M. Christ, J. Wright, {\it Multidimensional van der Corput and sublevel set estimates}, J. Amer. Math. Soc. {\bf 12} (1999), no. 4, 981-1015. 

\noindent [CKaNo] K. Cho, J. Kamimoto, T. Nose, {\it Asymptotic analysis of oscillatory integrals via the Newton polyhedra of the phase and the amplitude},  J. Math. Soc. Japan, {\bf 65} (2013), no. 2, 521-562.

\noindent [DeNiSa] J. Denef, J. Nicaise, P. Sargos, {\it  Oscillatory integrals and Newton polyhedra},  J. Anal. Math. {\bf 95} (2005), 147-172.

\noindent [D] J.J. Duistermaat, {\it Oscillatory integrals, Lagrange immersions, and unfolding of \hfill \break singularities}, Comm. Pure Appl. Math., {\bf 27} (1974), 207-281.

\noindent [ESa] L. Erd\H{o}s, M. Salmhofer, {\it Decay of the Fourier transform of surfaces with vanishing curvature}, (English summary) Math. Z. {\bf 257} (2007), no. 2, 261-294. 

\noindent [Gre] J. Green, {\it The role of structure in oscillatory integral estimates}, PhD thesis, University of Edinburgh, 2022.

\noindent [G1] M. Greenblatt, {\it Uniform bounds for Fourier transforms of surface measures in $\R^3$ with nonsmooth density},
Trans. Amer. Math. Soc. {\bf 368} (2016), no. 9, 6601-6625.  

\noindent [G2] M. Greenblatt, {\it Convolution kernels of 2D Fourier multipliers based on real analytic functions}, J. Geom. Anal. {\bf 28} (2018), no. 2, 787-816.

\noindent [G3] M. Greenblatt, {\it Fourier transforms of irregular mixed homogeneous hypersurface measures} (English summary),
Math. Nachr. {\bf 291} (2018), no. 7, 1075-1087.

\noindent [G4] M. Greenblatt, {\it Resolution of singularities, asymptotic expansions of oscillatory integrals, and related phenomena}, J. Anal. Math. {\bf 111}  (2010), no. 1, 221-245.

\noindent [G5] M. Greenblatt, {\it Oscillatory integral decay, sublevel set growth, and the Newton polyhedron}, Math. Annalen {\bf 346} (2010), no. 4, 857-895.

\noindent [G6] M. Greenblatt, {\it Van der Corput lemmas and Fourier transforms of irregular hypersurface measures}, preprint, arXiv:1409.4059v3.

\noindent [G7] M. Greenblatt, {\it Smooth and singular maximal averages over 2D hypersurfaces and associated Radon transforms}, Adv. Math. {\bf 377} (2021), Article ID 107465, 46 p.

\noindent [IKeM] I. Ikromov, M. Kempe, and D. M\"uller, {\it Estimates for maximal functions associated
to hypersurfaces in $R^3$ and related problems of harmonic analysis}, Acta Math. {\bf 204} (2010), no. 2,
151-271.

\noindent [IM] I. Ikromov, D. M\"uller, {\it  Uniform estimates for the Fourier transform of surface-carried measures in
$\R^3$ and an application to Fourier restriction}, J. Fourier Anal. Appl., {\bf 17} (2011), no. 6, 1292-1332.

\noindent [K1] V. N. Karpushkin, {\it A theorem concerning uniform estimates of oscillatory integrals when 
the phase is a function of two variables}, J. Soviet Math. {\bf 35} (1986), 2809-2826.

\noindent [K2] V. N. Karpushkin, {\it Uniform estimates of oscillatory integrals with parabolic or hyperbolic phases}, J. Soviet Math. {\bf 33} (1986), 1159-1188.

\noindent [PS] D. H. Phong, E. M. Stein, {\it The Newton polyhedron and oscillatory integral operators}, Acta Mathematica {\bf 179} (1997), 107-152.

\noindent [PrY] M. Pramanik, C.W. Yang, {\it Decay estimates for weighted oscillatory integrals in $\R^2$}, Indiana Univ. Math. J., {\bf 53}  (2004), 613-645.

\noindent [S] E. Stein, {\it Harmonic analysis; real-variable methods, orthogonality, and oscillatory \hfill\break
integrals}, Princeton Mathematics Series Vol. 43, Princeton University Press, Princeton, NJ, 1993.

\noindent [V] A. N. Varchenko, {\it Newton polyhedra and estimates of oscillatory integrals}, Functional Anal. Appl. {\bf 18} (1976), no. 3, 175-196.

\vskip 0.5 in

\noindent Department of Mathematics, Statistics, and Computer Science \hfill \break
\noindent University of Illinois at Chicago \hfill \break
\noindent 322 Science and Engineering Offices \hfill \break
\noindent 851 S. Morgan Street \hfill \break
\noindent Chicago, IL 60607-7045 \hfill \break
\noindent greenbla@uic.edu

\end{document}